\documentclass[12pt,letterpaper,reqno]{article} 

\usepackage{amsmath}
\usepackage{amsfonts}
\usepackage{amssymb}
\usepackage{amsthm}
\usepackage{amsrefs}
\usepackage{booktabs}
\usepackage{mathtools}
\usepackage{amscd}
\usepackage[margin=1.25in]{geometry}
\usepackage{enumitem}
\usepackage{tikz}
\usepackage[most]{tcolorbox}
\usepackage[colorlinks=true]{hyperref}
\usepackage[nosort,noabbrev,capitalise]{cleveref}

\usepackage{latexsym}
\usepackage{mathrsfs}
\usepackage{psfrag}
\usepackage[dvips]{epsfig}
\usepackage{epsfig}
\usepackage[all]{xy}         

\newcommand\Tstrut{\rule{0pt}{3ex}}         

\swapnumbers 

\theoremstyle{plain}                              
\newtheorem{thm}{Theorem}[section]
\newtheorem{prop}[thm]{Proposition}
\newtheorem{defn}[thm]{Definition}
\newtheorem{lem}[thm]{Lemma}
\newtheorem{cor}[thm]{Corollary}

\theoremstyle{definition}                         

\newtheorem*{eg}{Example}

\theoremstyle{remark}                             

\numberwithin{equation}{section}

\newcommand{\R}{\mathbb{R}}                     
\newcommand{\C}{\mathbb{C}}                     
\newcommand{\N}{\mathbb{N}}                     

\newcommand{\mcD}{\mathcal{D}}                     

\newtcolorbox{q1}[1][]{%
    colback=white,
    colframe=gray,
    boxrule=1pt,
    left=5pt,
    notitle,
    enhanced,
    breakable,
    }

\begin{document}
 
\title{\bf Sobolev Spaces and Their Applications to Partial Differential Equations}

\author{Shihan Kanungo}




\date{}
\maketitle

\setcounter{page}{1}
\begin{abstract}
\noindent
    This report provides an introduction to Sobolev spaces, a foundational concept in modern analysis and the theory of partial differential equations (PDEs). These spaces are useful to study, among other things, the well-posedness of partial differential equations and their approximation using finite elements. We begin with a historical overview, tracing the development of weak derivatives and the shift from classical to variational formulations of PDEs. After establishing the basic definitions and presenting key examples, we survey central theorems such as the Sobolev Embedding Theorem and Rellich's Theorem, emphasizing their significance in ensuring existence, uniqueness, and regularity of solutions. Finally, we discuss a classic application to PDEs, the Elliptic Regularity Theorem. We aim to provide a self-contained and accessible introduction for students with a background in real analysis and the theory of PDEs.
\end{abstract}

\section{Introduction}
In the 18th and 19th centuries, when PDE theory was being introduced, the focus was on classical solutions---functions smooth enough to satisfy differential equations pointwise. However, this approach soon revealed its limitations. There were many physically relevant problems, such as those arising in fuid dynamics, that admitted no smooth solutions (e.g. shock waves). The classical methods were not able to deal with this sort of situation.

The first step towards the framework of weak solutions was Dirichlet's Principle, developed by Riemann and others, which was the idea that the minimizer of a certain energy functional was the solution to the equation. This approach foreshadowed the variational formulation of PDEs.

The crucial paradigm shift happened in the 1930s with the work of Sergei Sobolev. In studying hyperbolic PDEs arising in mathematical physics and hydrodynamics, Sobolev realized that it was possible—and often necessary—to interpret derivatives in an \emph{integrated} sense rather than pointwise. He introduced what are now known as \emph{Sobolev spaces}, \( W^{m,p}(U) \), which consist of functions whose derivatives (in the weak sense) up to order \( m \) lie in the Lebesgue space \( L^p(U) \).

Sobolev’s innovation was twofold: first, to redefine the concept of a derivative so that it applied to functions not classically differentiable, and second, to endow the resulting function spaces with norms that rendered them complete and well-suited to functional analysis techniques. His work laid the groundwork for the existence and uniqueness theorems for weak solutions of PDEs and bridged the divide between pure and applied analysis.

These spaces play a central role in modern analysis, especially in the development of variational methods and the theory of elliptic, parabolic, and hyperbolic PDEs. They provide the natural functional framework for formulating and solving PDEs in a weak sense. This weak formulation is more flexible and more inclusive than classical approaches, and often leads to existence and uniqueness results under broader conditions.

This report aims to introduce the basic definitions and examples of Sobolev spaces, explore key theoretical results, and demonstrate their application to PDEs and related fields.

\section{Basic Definitions and Examples}

We start with some notation. Throughout, we will fix a positive integer $n$ and we will be working in the space $\R^n$. For a tuple $\alpha=(\alpha_1,\ldots,\alpha_n)$ of nonnegative integers, we define \[\partial^\alpha  = \frac{\partial^{\alpha_1+\cdots+\alpha_n}}{(\partial^{\alpha_1}x_1)\cdots (\partial^{\alpha_n}x_n)}.\] We set $|\alpha| = \alpha_1+\cdots +\alpha_n$ to be the degree of $\alpha$.

To define the Sobolev spaces, we first need to define \textit{distributions}.

Let $C_c^\infty$ denote the set of $C^\infty$ functions on $\R^n$ with \textit{compact support}, i.e. $\phi\in C_c^\infty$ if $\mathrm{supp}( \phi) = \{x:\phi(x)\ne 0\}$ is contained in a compact set. For an open set $U\subset \R^n$, we define $C_c^\infty(U)$ to be the set of functions in $C_c^\infty$ with support contained in $U$. 

A function $f\in L^p(U)$ can be identified with the map $F: C_c^\infty(U) \to \R$ given by $\phi \mapsto \int f\phi$. In fact, the map $F$ is a continuous linear functional. However, not all continuous linear functionals on $C_c^\infty(U)$ are of this form. We call a continuous linear functional on $C_c^\infty(U)$ a \textit{distribution} on $U$, and denote the space of distributions on $U$ by $\mcD'(U)$. Distributions are also known as, and can be thought of as,  ``generalized functions''. 

Every function $f$ such that $\int_C |f|<\infty$ for all compact sets $C$ defines a distribution, which we will also denote by $f$. Such functions are called \textit{locally integrable}. A distribution which corresponds to a function is called \textit{regular}.

As an example of a distribution which is not regular, we can define $\delta(\phi) = \phi(0)$. The map $\delta$ is called the \textit{delta function}, and is extremely important. More generally, the map $\phi\mapsto \partial^\alpha \phi(0)$ is a distribution for any $\alpha$. We will see later how these maps can be thought of as ``distributional derivatives'' of the delta function.

If $F$ is a distribution, we use $\langle F, \phi\rangle$ to denote the image of $\phi$ under $F$. We will also sometimes use
\[\int F(x) \phi(x),\]
even though $F$ is not a function. For example, we can write
\[\int \delta(x) \phi(x) = \phi(0).\]

We say that two distributions $F,G$ are equal ($F=G$) if $\langle F,\phi\rangle = \langle G,\phi\rangle$. We say $F=G$ on a open set $V$ if $\langle F,\phi\rangle = \langle G,\phi\rangle$ for all $\phi$ supported in $V$.

Clearly if $F=G$ on $V_1\cup V_2$, $F=G$ on both $V_1$ and $V_2$. The converse, though not immediately clear, nevertheless turns out to be true.

\begin{prop}
    Let $\{V_\alpha\}$ denote a collection of open subsets and let $V = \bigcup_\alpha V_\alpha$. If $F,G$ are distributions and $F=G$ on each $V_\alpha$, then $F=G$ on $V$.
\end{prop}

Due to this proposition, we can find a maximal open subset of $U$ such that $F=0$. We call the complement of this subset the \textbf{support} of $F$. Note that we cannot define the support of $F$ as the set of points $x$ with $F(x)\ne 0$, since $F$ is a distribution, not a function.

We can extend many operations on functions to distributions via the following process.

Let $T$ be a linear operator on some subspace of the locally integrable functions on an open set $U$. Suppose also, that there exists another continuous linear operator $T': C_c^\infty(U)\to C_c^\infty(U)$ such that
\[\int (Tf)\phi = \int f (T'\phi)\]
for all $f$ in the domain of $T$ and $\phi\in C_c^\infty$. (Note that the definition of the operator $T'$ is similar to the definition of an adjoint). We can then extend the map $T$ to the space of distributions by defining \[\langle TF,\phi\rangle=\langle F,T'\phi\rangle.\] Since $T'$ is continuous, the extension of $T$ is also guaranteed to be complete. 

We now discuss some important examples of this process.

First, define $Tf = \partial^\alpha f$. Then, for $\phi \in C_c^\infty(U)$, we can use integration by parts, to get \[\int (\phi^\alpha f)\phi = (-1)^{|\alpha|}\int f (\partial^\alpha \phi).\] Thus $T' = (-1)^\alpha \partial^\alpha$. For an arbitrary distribution $F$, we can define the \textbf{derivative} of $F$ by
\[\langle \partial^\alpha F,\phi\rangle = (-1)^{|\alpha|} \langle F,\partial^\alpha\phi\rangle.\]
This is a very powerful tool, because it allows us to define derivatives of arbitrary functions that are not differentiable in the classical sense, or not even continuous at all. This is going to be the main idea behind Sobolev spaces.

We can also multiply $F$ by a smooth function $\psi\in C^\infty(U)$; in this case $Tf=\psi f$, $T' = T$, and
\[\langle \psi F,\phi\rangle = \langle F, \psi\phi\rangle.\]

We are now ready to define the Sobolev Spaces. For $m\in \N$ and $1\le p \le \infty$, the \textbf{Sobolev Space} $W^{m,p}(U)$ consists of functions $u\in L^p(U)$ whose derivatives $\partial^\alpha u$ for $|\alpha|\le m$, \textit{in the distributional sense}, are regular and correspond to functions in $L^p(U)$. The parameter $m$ measures ``how differentiable'' functions are, analogously to $C^m(U)$. 

An example of why Sobolev Spaces are useful is as follows.

\begin{eg}
    The Sobolev space $W^{1,2}(U)$ consists of the $L^2$ functions on $U$ with a $L^2$ first derivative. If $U = (-1,1)$, then the function $f=|x|$ is an element of $W^{1,2}(U)$. To see this, note that $f$ is square-integrable, and it is easily verified that the (distributional) derivative of $f$ regular and corresponds to the step function
    \[f'(x) = \begin{cases}
        1 & x>0\\
        -1 & x<0
    \end{cases},\]
    which is also $L^2$. Note that we do not need to define $f'(0)$ because we are considering $f'$ as an element of $L^2(U)$. The function $f$ has has no derivative in the classical sense, but using distributions we were able to define $f'$ in a way that behaves exactly like the classical derivative.
\end{eg}

\begin{eg}
    What is the second derivative of the function $f$ above? Using integration by parts, we have
    \[\int_{-1}^1 f''(x) \phi(x) = - \int_{-1}^1 f'(x)\phi'(x) = \int_{-1}^0 \phi'(x) - \int_{0}^1 \phi(x) = 2\phi(0).\]
    (The boundary terms vanish since $\phi$ is supported in $(-1,1)$). Then it follows that $\partial^2 f=2\delta$, which doesn't correspond to any regular function. Thus, $f$ would not be an element of the Sobolev space $W^{2,2}(U)$.
\end{eg}

Clearly, $W^{m,p}(U)$ is a vector space. Additionally,
\[\Vert u\Vert_{m,p,U} = \left(\sum_{|\alpha|\le m} \int_U|\partial^\alpha u|^p\right)^{1/p}\]
defines a norm on $W^{m,p}(U)$ for $1\le p <\infty$ and for $p=\infty$, the norm is given by
\[\Vert u \Vert_{m,\infty, U} = \max_{0\le |\alpha|\le m}\Vert \partial^\alpha u\Vert_\infty.\]

Just like how $L^p(U)$ is an inner product space when $p=2$, so also is $W^{m,2}(U)$.  We use the notation $H^m(U) \coloneqq W^{m,2}(U)$; the inner product is given by
\[\langle u,v\rangle_{m,U} = \sum_{|\alpha|\le m}\int_U \partial^\alpha u(x) \overline{\partial^\alpha v(x)}.\]
It is easy to see that this inner product matches the norm defined earlier.

The Fourier transform converts derivatives into algebraic operations. Because of this, it will be useful to give an alternate definition of $H^m(\R^n)$ using the Fourier transform.

Our notation for the Fourier transform of $f$ will be $\widehat{f}(\xi)$, where $\xi\in \R^N$. We write $\xi^\alpha = \xi_1^{\alpha_1}\cdots \xi_n^{\alpha_n}$, similarly to our notation $\partial^\alpha$. Using basic results about the Fourier transform, it is not hard to check, that $f\in H^m$ if and only if $\xi^\alpha \widehat{f}\in L^2$ for all $|\alpha|\le m$. Additionally, there exist positive constants $C_1,C_2$ such that
\[C_1(1+|\xi|^2)^m \le \sum_{|\alpha|\le m}|\xi^\alpha|^2 \le C_2 (1+ |\xi|^2)^m.\]
To see this, note that when $(1+|\xi|^2)^k$ and $ \sum_{|\alpha|\le m}|\xi^\alpha|^2$ are expanded, they have the same terms, only with different coefficients, which depend only on $m$. From here, the statement is clear.
It then follows that $f\in H^m$ if and only if $(1+|\xi|^2)^{m/2}\widehat{f}$ is in $L^2$. Additionally, it follows that the norm
\[\Vert f\Vert_{(m)}\mapsto \left\lVert (1+|\xi|^2)^{m/2}\widehat{f}\right\rVert_2\]
is equivalent to the norm $\Vert\cdot \Vert_{m,2, \R^n}$ defined above. However, this norm can be extended to arbitrary values of $m\in \R$, and hence we will use this norm in the definition of the Sobolev Space.

\begin{defn}
    For $m\in \R$, the \textbf{Sobolev Space} $H^m$ is defined by
    \[H^m = \{f\in L^2 (\R): (1+ |\xi|^2)^{m/2}\widehat{f}\in L^2\}.\]
\end{defn}
As seen above, this definition matches with our original definition when $p=2$. 

\section{The Sobolev Embedding Theorem}
We can think of the Sobolev space $W^m(U)$ as the $L^2$ functions that are ``$L^2$-differentiable'' up to order $m$. This does not coincide with the ordinary notion of smoothness: for example the function $|x|$ is ``$L^2$-differentiable'' up to order $1$, but is not $C^1$. Thus, it makes sense to ask the following question: what can we say about the smoothness of elements of the Sobolev space $W^m(U)$?

This question is answered by the Sobolev Embedding Theorem, and the answer turns out to be quite nice.

We let $C_0$ denote the space of continuous functions $f$ on $\R^n$ to vanish at $\infty$; i.e. the space $\{x: |f(x)|>\varepsilon\}$ is compact for any $\varepsilon>0$. We let
\[C_0^k = \{f\in C_0: \partial^\alpha f \in C_0 \quad\forall |\alpha|<k\},\]
which means that $f$ has all (classical) derivatives up to order $k$ continuous and vanishing at $\infty$.
\begin{thm}[Sobolev Embedding Theorem]
    Suppose $m>k+n/2$. Then,
    \begin{enumerate}
        \item[$(1)$] If $f\in H^m$, then $\widehat{\partial^\alpha f}\in L^1$ and $\Vert \widehat{\partial^\alpha f}\Vert_1 \le C \Vert f \Vert_{(m)}$ for all $|\alpha|\le k$, where $C$ is a constant depending only on $m-k$.
        \item[$(2)$] $H^m \subset C_0^k$ and the inclusion map is continuous.
    \end{enumerate}
\end{thm}
\begin{proof}
    For $f\in H^m$ and $|\alpha|\le k$, we have
    \begin{align*}
        (2\pi)^{-|\alpha|} \Vert \widehat{\partial^\alpha f}\Vert_1 &=(2\pi)^{-|\alpha|}\int |(\widehat{\partial^\alpha f})(\xi) |\, d\xi\\  &= \int |\xi^\alpha \widehat{f}(\xi)|\, d\xi \\
        &\le \int (1+|\xi|^2)^{k/2} |\widehat{f}(\xi)|\, d\xi.
    \end{align*}
    This is because $|\alpha| \le k$ implies $|\xi|^{2\alpha}$ is a term in the expanded form of $(1+|\xi|^2)^k$, so $|\xi^\alpha|\le (1+|\xi|^2)^{k/2}$. Now applying the Cauchy-Schwarz inequality,
    \begin{align*}
        \int (1+|\xi|^2)^{k/2} |\widehat{f}(\xi)| \, d\xi &\le \left( \int (1+|\xi|^2)^{m} |\widehat{f}(\xi)|^2\, d\xi\right)^{1/2}\left( \int (1+|\xi|^2)^{k-m}\, d\xi\right)^{1/2}\\
        &= \Vert f \Vert_{(m)}\left( \int (1+|\xi|^2)^{k-m}\, d\xi\right)^{1/2}.
    \end{align*}
    Since $m-k< - n/2$, the last integral above is finite. Combining everything together, we have proven claim (1). The \textit{Riemann-Lebesgue Lemma} says that the Fourier transform of a $L^1$ function belongs to $C_0$. Using the Fourier inversion theorem, it follows that $\partial^\alpha f\in C_0$. From here, claim (2) follows.  
\end{proof}

What the Sobolev Embedding Theorem implies is that if $f\in H^m$, then $f$ must have some degree of smoothness; namely it must have at least $\lceil m-\frac n2 \rceil -1$ continuous derivatives.

For example, if $m=1$ and $n=1$, which means that $f$ has one ``$L^2$-derivative'', then $f$ must be in $C_0^0=C_0$, i.e. $f$ must be continuous. Thus, while elements of $H^1(\R)$ may not be differentiable in the classical sense, we know that they must be continuous. 

Another immediate corollary of the Sobolev Embedding theorem is as follows.
\begin{cor}\label{cor: infinite L2 differentiable implies Cinfty}
    If $f\in H^m$ for all $m$, then $f\in C^\infty$. 
\end{cor}
\begin{proof}
    To see this, note that for any $k\in \N$, $f\in H^{k+n/2+1}$, so by the Sobolev Embedding Theorem, $f\in C_0^k$. Hence $f\in C^\infty$.
\end{proof}
The Sobolev Embedding Theorem is a very useful tool for solving PDEs, because it lets you ``upgrade'' a function $u\in H^m$ to a function $u\in C_0^k$.

Under some conditions, multiplication by suitably smooth functions preserves the Sobolev space $H^m$.
\begin{thm}\label{thm: multiply by smooth functions}
    Suppose $\phi \in C_0$ and 
    \[\int (1+|\xi|^2)^{a/2} |\widehat{\phi}(\xi)| \, d\xi\]
    is finite, for some $a>0$. Then for all $|s|\le a$, the map $f\mapsto \phi f$ is a bounded, continuous operator on $H^m$.
\end{thm}

Next, we have an extremely important result about Sobolev spaces, which can be thought of as a version of the Bolzano-Weierstrass theorem. 

\section{Rellich's Theorem}

\begin{thm}[Rellich's Theorem]
    Suppose that $\{f_k\}$ is a sequence of distributions in $H^m$ that are all supported in a fixed compact set $K$. Additionally, assume that the sequence $\Vert f_k \Vert_{(m)}$ is bounded. Then there exists a subsequence $\{f_{k_j}\}$ that converges in $H^s$ for all $s<m$.
\end{thm}
Rellich's Theorem is important because it lets us turn weak convergence into strong convergence, which allows as to take limits and prove existence of solutions in PDEs.

We define the \textbf{localized Sobolev space} $H_\mathrm{loc}^m(U)$ to be the set of all distributions $f$ on $U$ such that for every precompact open set $V$ with $\overline{V}\subset U$, there exists $g\in H^m$ such that $g=f$ on $V$.

\begin{prop}\label{prop: localized sobolev space}
    A distribution $f$ on $U$ is in $H_{\mathrm{loc}}^m(U)$ if and only if $\phi f\in H^m$ for every $\phi\in C_c^\infty(U)$.
\end{prop}
\begin{proof}
    If $f\in H_{\mathrm{loc}}^m(U)$ and $\phi\in C_c^\infty(U)$, then there exists $g\in H^s$ such that $f$ agrees with $g$ on a neighborhood of $\mathrm{supp}(\phi)$. It follows that $\phi f= \phi g \in H^m$, since $\phi$ satisfies the conditions of \cref{thm: multiply by smooth functions} for any $a$. 

    For the other direction, given a precompact open set $V$ with $\overline{V}\subset U$, there exists $\phi\in C_c^\infty(U)$ with $\phi=1$ on a neighborhood of $\overline{V}$. Then, setting $g=\phi f$, $g=\phi f\in H^m$ and $g = f$ on $V$, so $f\in H_{\mathrm{loc}}^m(U)$.
\end{proof}
\section{Application to PDEs}
Sobolev spaces arise naturally in the weak formulation of PDEs.

Consider the classical Poisson problem:
\[
- \Delta u = f \quad \text{in } \Omega, \qquad u = 0 \quad \text{on } \partial \Omega,
\]
where $\Omega \subset \mathbb{R}^n$ is a bounded open domain with smooth boundary, and $f \in L^2(\Omega)$. A classical solution requires $u \in C^2(\Omega) \cap C(\overline{\Omega})$, which may be too restrictive. We instead aim for a \textit{weak solution}.

The first step is to multiply by a test function $\varphi\in C_c^\infty(\Omega)$
\[
-\Delta u \cdot \varphi = f \cdot \varphi.
\]
Integrating over $\Omega$, we obtain
\[
\int_\Omega -\Delta u \cdot \varphi \, dx = \int_\Omega f \cdot \varphi \, dx.
\]
Using Green's identity (with $\varphi = 0$ on $\partial \Omega$), we have
\[
\int_\Omega \nabla u \cdot \nabla \varphi \, dx = \int_\Omega f \cdot \varphi \, dx.
\]
This relation is satisfied by any classical solution to the Poisson problem. But this now lets us define a broader notion of a solution to the equation, using the Sobolev space $H_0^1(\Omega)$ of $L^2$ functions with $L^2$ first derivatives and zero trace on the boundary.

Precisely, we say that $u \in H_0^1(\Omega)$ is a \textit{weak solution} to the Poisson equation if
\[
\int_\Omega \nabla u \cdot \nabla v \, dx = \int_\Omega f \cdot v \, dx \quad \forall v \in H_0^1(\Omega).
\]

Now, we claim that this formulation is a well-posed problem.

To see this, let \[a(u,v) = \int_\Omega \nabla u \cdot \nabla v \, dx\quad \text{and} \quad \ell(v) = \int_\Omega f \cdot v \, dx.\]

Then $a$ is coercive and continuous on $H_0^1(\Omega)$, and $\ell$ is bounded. By the Lax--Milgram theorem, there exists a unique $u \in H_0^1(\Omega)$ solving the problem.

Thus, the classical PDE is reformulated into a well-posed variational problem in a Sobolev space.

Extending this idea, we now consider the Poisson equation with Neumann boundary conditions.
That is,
\[
- \Delta u = f \quad \text{in } \Omega, \qquad \frac{\partial u}{\partial n} = 0 \quad \text{on } \partial \Omega,
\]
with $f \in L^2(\Omega)$ and $\Omega$ smooth and bounded.

Again, we multiply by a test function $\varphi\in C^{\infty}(\overline{\Omega})$:
\[
-\Delta u \cdot \varphi = f \cdot \varphi,
\]
and then we integrate.
\[
\int_\Omega -\Delta u \cdot \varphi \, dx = \int_\Omega f \cdot \varphi \, dx.
\]
Green's identity gives
\[
\int_\Omega \nabla u \cdot \nabla \varphi \, dx - \int_{\partial \Omega} \frac{\partial u}{\partial n} \cdot \varphi \, dS = \int_\Omega f \cdot \varphi \, dx.
\]
But since $\partial u / \partial n = 0$ on $\partial \Omega$, the boundary term vanishes, and we obtain
\[
\int_\Omega \nabla u \cdot \nabla \varphi \, dx = \int_\Omega f \cdot \varphi \, dx.
\]
This is just like before, and we define the weak formulation in the same way: $u \in H^1(\Omega)$ is a \textit{weak solution} if
\[
\int_\Omega \nabla u \cdot \nabla v \, dx = \int_\Omega f \cdot v \, dx \quad \forall v \in H^1(\Omega).
\]
Note that here, $v$ is not required to vanish at the boundary.

This Neumann problem is solvable only if we have
\[
\int_\Omega f \, dx = 0.
\]
Moreover, the solution is unique \textit{up to a constant}.
This also follows from the Lax-Milgram theorem: let
\[
V := \left\{ v \in H^1(\Omega) \mid \int_\Omega v \, dx = 0 \right\}.
\]
On $V$, the bilinear form
\[
a(u,v) = \int_\Omega \nabla u \cdot \nabla v \, dx
\]
is coercive and continuous. The functional $\ell(v) = \int_\Omega f \cdot v \, dx$ is bounded. Then the Lax-Milgram theorem ensures a unique solution $u \in V$.

We have the following comparisons to to the Dirichlet case:

\begin{center}
\begin{tabular}{ccc}
\toprule
\textbf{Feature} & \textbf{Dirichlet BC} & \textbf{Neumann BC} \\ [0.5ex]
\hline
Function Space & $H_0^1(\Omega)$ & $H^1(\Omega)$\Tstrut \\ [1ex]
Boundary & $u = 0$ & $\partial u/\partial n = 0$ \\ [1ex]
Uniqueness & Unique & Up to constant \\ [1ex]
Compatibility & Not needed & $\int_\Omega f = 0$ \\ [1ex]
\bottomrule
\end{tabular}
\end{center}

\section{The Elliptic Regularity Theorem}
Having explored how Sobolev spaces facilitate the formulation and analysis of boundary value problems such as the Poisson equation with Dirichlet and Neumann conditions, we now turn to a fundamental aspect of weak solutions: their regularity. In this section, we present the elliptic regularity lemma, which characterizes the smoothness properties of solutions to elliptic PDEs in Sobolev spaces.

First, consider the ODE case. Let \[L= \sum_0^m a_j\left(\frac d{dx}\right)^j\] be a differential operator, where each $a_j$ is a $C^\infty$ function of $x$. Furthermore, assume that $a_m$ never vanishes. Then, it is easy to prove that smooth data give smooth solutions. More precisely, if $Lu=f$ and $f$ is $C^k$ on some interval $I$, $u$ is $C^{k+m}$ on $I$.

However, such smoothness results do not hold in general for PDEs. As an example, consider the wave equation $u_{tt}-u_{xx}=0$. For any locally integrable function $f$, $u=f(x-t)$ is a (weak) solution of the equation, but $u$ only has as much smoothness as $f$. What the Elliptic Regularity Theorem aims to do is identify a large class of PDEs for which we can prove a strong regularity theorem similar to the ODE case.

Let \[P(\partial) = \sum_{|\alpha|\le m} c_\alpha \partial^\alpha\] be a partial differential operator with constant coefficients $c_\alpha$. Furthermore, we assume that $c_\alpha \ne 0$ for some $\alpha$ with $|\alpha| = m$, which means that $P$ is order $m$. We define the \textit{principal symbol} $P_m$ to be
\[P_m(\xi) = \sum_{|\alpha| = m} c_\alpha (i\xi)^\alpha.\]
We also write
\[P(\xi) = \sum_{|\alpha| \le m} c_\alpha (i\xi)^\alpha.\]
In other words, we replace the $\partial$ in the definition of $P(\partial)$ with $i\xi$; and we can think of $P(\xi)$ as the Fourier transform of $P(\partial)$.
We call $P$ \textit{elliptic} if $P_m(\xi)\ne 0$ for all nonzero $\xi\in \R^n$. Intuitively, we can think of ellipticity meaning that $P$ is ``$m$th order in all directions''. For example, the Laplacian $\Delta$ is elliptic because the principal symbol is
\[-(\xi_1^2 + \cdots + \xi_n^2) = - |\xi|^2,\]
while the heat operator $\partial_t - \Delta$ is not because the principal symbol is
\[\xi_1^2 +\cdots + \xi_{n}^2,\]
which is not nonzero for all nonzero $(\xi,\tau)\in \R^{n+1}$ (for example $\xi = (0,0,\ldots, 0)$ and $\tau = 1$). Similarly, the wave operator is not elliptic.

\begin{lem}\label{lem: elliptic symbols}
    Suppose $P$ is order $m$. Then $P$ is elliptic if and only if there exist constants $C,R>0$ such that $|P(\xi)| \ge C |\xi|^m$ for all $|\xi| \ge R$.
\end{lem}
The proof of this lemma is relatively straightforward. Using this, we get another lemma.
\begin{lem}\label{lem: regularity for all of R}
    If $P(\partial)$ is elliptic of order $m$ and $u\in H^s$ and $P(\partial)u \in H^s$, then $u\in H^{s+m}$.
\end{lem}
\begin{proof}
    We know that $(1+|\xi|^2)^{s/2}\widehat{u}\in L^2$ and $(1+|\xi|^2)^{s/2}P(\xi) \widehat{u}\in L^2$. By the previous lemma, for some $R\ge 1$ we have
    \[(1+|\xi|^2)^{m/2} \le 2^m |\xi|^m \le C^{-1}2^m |P(\xi)|\]
    for all $|\xi| \ge R$. For $|\xi|\le R$, $(1+|\xi|^2)^{m/2}$ is bounded by some constant. Thus,
    \[(1+|\xi|^2)^{(s+m)/2}|\widehat{u}| \le C' (1+|\xi|^2)^{s/2}(|P(\xi) \widehat{u}|+|\widehat{u}|)\in L^2\]
    for some constant $C'$. This implies $u\in H^{s+m}$. 
\end{proof}
The previous lemma is nice, but it only applies to situations where the domain is all of $\R^n$ and the functions vanish at infinity. We would like to be able to deal with different domains, and that is precisely the content of the Elliptic Regularity Theorem. The idea is that we replace $H^s$ with the localized Sobolev space $H_\mathrm{loc}^s(\Omega)$.
\begin{thm}[Elliptic Regularity Theorem]\label{thm: ellitic regularity}
    Suppose that $L$ is a constant-coefficient elliptic differential operator of order $m$, $\Omega$ is an open set in $\R^n$ and $u$ is a distribution on $\Omega$. If $Lu \in H^s_{\mathrm{loc}}(\Omega)$ for some $m\in \R$, then $u\in H_{\mathrm{loc}}^{s+m}(\Omega)$.
\end{thm}
\begin{proof}[Sketch of Proof]
    Using \cref{prop: localized sobolev space}, it suffices to show that $\phi u \in H^{s+m}$ for all $\phi\in C_c^\infty(\Omega)$. We can find a precompact open set $V$ such that $\phi$ is supported in $V$ and $\overline{V}\subset \Omega$. Then, pick $\psi\in C_c^\infty(\Omega)$ with $\psi=1$ on $\overline{V}$. 

    We can then find that $\psi u \in H^{\sigma}$ for some $\sigma\in \R$. Since $H^{x}\subset H^y$ for $x>y$, we may decrease $\sigma$ such that $k \coloneqq s+m-\sigma$ is a positive integer. We then choose functions $\psi = \psi_0,\psi_1,\ldots,\psi_k=\phi$, such that:
    \begin{itemize}
        \item $\psi_1,\ldots,\psi_{k-1}\in C_c^{\infty}$,
        \item $\psi_j = 1$ on a neighborhood of $\mathrm{supp}(\psi)$ for $j\le k-1$,
        \item $\psi_j$ is supported in $\{x: \psi_{j=1}\}$.
    \end{itemize}
    Intuitively, we can visualize the $\psi_j$ as getting closer and closer to the function that is $1$ on the support of $\phi$ and $0$ outside the support of $\phi$.

    We prove by induction that $\psi_j u \in H^{\sigma + j}$.

    To prove this, we use the following crucial observation.
    For any $\zeta \in C_c^\infty$, if we define
    \[[L,\zeta]f = L(\zeta f) - \zeta L f,\]
    then by using the product rule for derivatives on the first term, we see that the $m$th order derivatives of $f$ actually cancel, so $[L,\zeta]$ is a $(m-1)$th order operator! Thus if $f\in H^t$, it follows that $[L,\zeta]f\in H^{t-(m-1)}$. 

    Now we can do the induction. The base case $\psi_0 u \in H^{\sigma}$ follows from the definition of $\sigma$. Now suppose that $\psi_j u \in H^{\sigma + j}$ for some $j<k$. Then, we have
    \begin{align*}
        L(\psi_{j+1}u) = \psi_{j+1}Lu + [L,\psi_{j+1}]u 
    \end{align*}
    By the definition of the $\psi_j$, $\psi_{j+1}$ is nonzero only when $\psi_j =1$. It follows that we can replace the $u$ in $[L,\psi_{j+1}]u$ with $\psi_j u$ without changing anything. Since $\psi_j u \in H^{\sigma + j}$, $[L,\psi_{j+1}]\psi_j u \in H^{\sigma + j + 1-m}$. We also have by assumption, $\psi_{j+1}Lu \in H^s$. Thus, $L(\psi_{j+1}u) \in H^{\sigma+j+ 1 -m}$. Finally, $\psi_{j+1}u = \psi_{j+1}\psi_j u \in H^{\sigma + j}$. Using \cref{lem: regularity for all of R}, it follows that $\psi_{j+1}u\in H^{\sigma + j+1}$, proving the inductive hypothesis.

    It follows that when $j=k$, $\phi u \in H^{\sigma+k}=H^{s+m}$. Since $\phi$ was an arbitrary element of $C_c^\infty(\Omega)$, it follows that $u\in H_\mathrm{loc}^{s+m}(\Omega)$, and we are done.
\end{proof}
This theorem has many important corollaries.
\begin{cor}\label{cor: Lu = Cinfty implies u is Cinfty}
    Suppose that $L$ is a constant-coefficient elliptic differential operator of order $m$, $\Omega$ is an open set in $\R^n$, and $u$ is a distribution on $\Omega$. If $Lu\in C^\infty(\Omega)$, $u\in C^\infty(\Omega)$.
\end{cor}
\begin{proof}
    Since $Lu\in C^\infty(\Omega)$, $Lu\in H^{\mathrm{loc}}_{s}(\Omega)$ for all $s$, so $u\in H^{\mathrm{loc}}_{s+m}(\Omega)$ for all $s$. Then by \cref{cor: infinite L2 differentiable implies Cinfty}, $u\in C^\infty(\Omega)$.
\end{proof}
As an example of this, consider the Laplace equation $\Delta u =0$. Since $\Delta$ is elliptic and $0\in C^\infty$, we get that any distributional solution must also be $C^\infty$. 

Secondly, consider the equation $Lu = 0$, where $L = \partial_1 + i\partial_2$ on $\R^2$, known as the \textit{Cauchy-Riemann} equation. Considering $u(x_1,x_2)$ as a function on $\C$: $u(z)=u(x_1+ix_2)$, the Cauchy-Riemann tells us if $u$ is \textit{holomorphic}, i.e. it is complex differentiable. By \cref{cor: Lu = Cinfty implies u is Cinfty}, it follows that all holomorphic functions are $C^\infty$, i.e. they are infinitely differentiable. We have just recovered a major result in complex analysis!

\section{A Regularity Theorem for the Heat Equation}
The heat operator $\partial_t - \Delta$ is not elliptic, so the Elliptic Regularity theorem does not hold. However, using similar arguments, we can prove a weaker version of it.

We are working in $\R^{n+1}$ with the coordinated $(x,t)$ and dual coordinates $(\xi,\tau)$. As before, we set $P(\partial) = \partial_t - \Delta$ and $P(\xi,\tau) = i\tau + |\xi|^2$. 

\begin{lem}
    There exist positive $C,R$ such that $|\xi| |(\xi,\tau)|^2 \le C |P(\xi,\tau)|$ for $|(\xi,\tau)|>R$.
\end{lem}
The proof is similar to the proof of \cref{lem: elliptic symbols}. Next, we have an analog of \cref{lem: regularity for all of R}.
\begin{lem}\label{lem: regularity for heat all of Rd}
    If $u, P(\partial) u\in H^2$, then $u\in H^{s+1}$ and $\partial_i u \in H^{s+1/2}$ for $1\le i\le n$ (Here $\partial_i$ denotes partial differentiation with respect to $x_i$).
\end{lem}
\begin{proof}
    We know that $(1+|(\xi,\tau)|^2)^{s/2}\widehat{u}\in L^2$ and $(1+|(\xi,\tau)|^2)^{s/2}P(\xi,\tau) \widehat{u}\in L^2$. By the previous lemma, for some $R\ge 1$ we have
    \[(1+|(\xi,\tau)|^2)^{1/2} \le 2 R|(\xi,\tau)| \le 2|\xi| |(\xi,\tau)|^2 \le C^{-1}2 |P(\xi,\tau)|\]
    for all $|\xi|>R$, since this means $|(\xi,\tau)|>R$. If $|\xi|\le R$, then $(1+|(\xi,\tau)|^2)^{1/2} \le (D+|\tau|^2)^{1/2}$ for a constant $D$. For $|\tau|>1$, we can bound this as $E |\tau|$ for some constant $E$. For $|\tau| \le 1$, this is bounded above by a constant $E'$. In the second case, we get $(1+|(\xi,\tau)|^2)^{1/2}<E'$. In the first case, note that $|P(\xi,\tau)| = |i\tau + |\xi|^2|>|\tau|$, so $(1+|(\xi,\tau)|^2)^{1/2} < E |P(\xi,\tau)|$.
    Combining everything together,
    \[(1+|(\xi,\tau)|^2)^{(s+1)/2}|\widehat{u}| \le C' (1+|(\xi,\tau)|^2)^{s/2}(|P(\xi,\tau) \widehat{u}|+|\widehat{u}|)\in L^2\]
    for some constant $C'$. This implies $u\in H^{s+1}$. Next, for any $1\le i\le n$,
    \[(1+|(\xi,\tau)|^2)^{(s+(1/2))/2} |\widehat{\partial_i u}| \le (1+|(\xi,\tau)|^2)^{(s+(1/2))2} |\xi | |\widehat{u}|.\]
    Now if $|(\xi,\tau)| >R$, 
    \[(1+|(\xi,\tau)|^2)^{1/4} |\xi|\le F |\xi| |(\xi,\tau)|^2\ \le DC^{-1} |P(\xi,\tau)|,\]
    for some constant $F$. If $|(\xi,\tau)|<R$, then
    \[(1+|(\xi,\tau)|^2)^{1/4} |\xi| <G\]
    for a constant $G$. Hence
    \[(1+|(\xi,\tau)|^2)^{(s+(1/2))2} |\xi | |\widehat{u}| \le C'' (1+|(\xi,\tau)|^2)^{s/2}(|P(\xi,\tau) \widehat{u}|+|\widehat{u}|)\]
    for a constant $C''$, implying that $\partial_i u\in H^{s+1/2}$.
\end{proof}

\begin{thm}[Regularity Theorem for the Heat Equation] 
    Let $\Omega$ be an open subset of $\R^{n+1}$, $u$ a distribution on $\Omega$, and $(\partial_t - \Delta)u \in H_\mathrm{loc}^{s}(\Omega)$. Then $u\in H_{\mathrm{loc}}^{s+1}(\Omega)$.
\end{thm}
\begin{proof}[Sketch of Proof]
    We mimic the proof of \cref{thm: ellitic regularity}. Again, it suffices to show that $\phi u \in H^{s+1}$ for all $\phi\in C_c^\infty(\Omega)$. Define $V$, $\psi$, as before. We have $\psi u \in H^\sigma$ for some $\sigma \in \R$; we can decrease $\sigma$ so that $k \coloneqq s+1-\sigma$ is an integer. Additionally; we can assume $\partial_i (\psi u) \in H^{\sigma -1/2}$ for all $i=1,\ldots, n$. Again, choose functions $\psi_j$ as before, except this time we have $2k+1$ functions $\psi_0,\ldots, \psi_{2k}=\phi$ rather than $k+1$. 
    
    We induct on $j$ to show $\psi_j u \in H^{\sigma + (j/2)}$ and $\partial_{i}(\psi_j u)\in H^{\sigma + (j-1)/2}$ for all $i=1,\ldots, n$. 
    
    For the base case, $\psi_0 u \in H^{\sigma}$ and $\partial_i (\psi_j u)\in H^{\sigma -1/2}$ by hypothesis. Now suppose $\psi_j u \in H^{\sigma + (j/2)}$ and $\partial_{i}(\psi_j u)\in H^{\sigma + (j-1)/2}$ for all $i=1,\ldots, n$, and $j<2k$. It is easily computed using the product rule that
    \[[\partial_t - \Delta, \zeta] u = (\partial_t \zeta - \Delta \zeta)u - 2 \sum_{i=1}^n (\partial_i \zeta)(\partial_i u).\]

    Then, we have
    \[L(\psi_{j+1}u) = \psi_{j+1} Lu + [L,\psi_{j+1}]u = \psi_{j+1}Lu + [L,\psi_{j+1}]\psi_j u.\]
    Then, we compute that
    \[[L,\psi_{j+1}]\psi_j u = (\partial_t \psi_{j+1}-\Delta \psi_{j+1})(\psi_j u)-2\sum_{i=1}^n (\partial_i \psi_{j+1})(\partial_i \psi_j u).\]
    By the inductive hypothesis, the first term is in $H^{\sigma + (j/2)}$ and the second is in $H^{\sigma + (j-1)/2}$. Also, $\psi_{j+1}Lu \in H^s$ by assumption. Note that $\sigma + (j-1)/2\le s$, so $L(\psi_{j+1}u) \in H^{\sigma + (j-1)/2}$. We also have $\psi_{j+1}u = \psi_{j+1}\psi_j u \in H^{\sigma + j/2}$. Thus, by \cref{lem: regularity for heat all of Rd}, $\psi_{j+1}u \in H^{\sigma + (j+1)/2}$ and $\partial_i (\psi_{j+1}u)\in H^{\sigma + j/2}$. This completes the inductive step.

    Thus, we get $\phi u \in H^{s+1}$, and as before, we are done.
\end{proof}
Note that this theorem is not as good as the Elliptic Regularity theorem, because the heat equation is second order, and we can only say that it has one extra derivative.

\section*{Conclusion}
Sobolev spaces form the backbone of modern analysis and partial differential equations. They generalize classical differentiability, allow for broader classes of solutions, and provide a unifying framework for both theoretical and applied work. This report aimed to introduce these spaces and highlight their critical role in PDE theory. For further details and rigorous proofs, we refer the reader to the standard references.

\vspace{10mm}

\vfill

\rule{5cm}{1pt}

{\footnotesize \sc Department of Mathematics, San Jos\'e State University, San Jos\'e, CA 95192}


{\footnotesize \em Email address: } \texttt{\small shihan.kanungo@sjsu.edu}

\end{document}